\documentclass[12pt,reqno]{amsart}
\usepackage[utf8]{inputenc}
\usepackage{amsmath,amssymb,amsthm,color}
\usepackage{graphicx}
\usepackage[margin=2.8cm,top=3.5cm,bottom=3.5cm,footskip=1cm,headsep=1cm]{geometry}

\def\div{\mathop{\rm div}}
\def\eps{\varepsilon}

\def\d{d}
\def\div{{\rm div}}
\def\dt{\partial_t}
\def\ddt{\tfrac{d}{dt}}

\def\RR{\mathbb{R}}

\def\div{{\rm div}}

\newtheorem{theorem}{Theorem}
\newtheorem{lemma}{Lemma}
\newtheorem{corollary}{Corollary}
\newtheorem{assumption}{Assumption}

\begin{document}
\title[Uniqueness of nonlinear diffusion coefficients]{On the uniqueness of nonlinear diffusion coefficients in the presence of lower order terms}
\author[H. Egger]{Herbert Egger$^\dag$}
\thanks{$^\dag$AG Numerical Analysis and Scientific Computing, Department of Mathematics, TU Darmstadt, Dolivostr. 15, 64293 Darmstadt, Germany. e-mail: {\tt egger@mathematik.tu-darmstadt.de} }
\author[J.-F. Pietschmann]{Jan-Frederik Pietschmann$^*$}
\thanks{$^*$Institut f\"ur Numerische und Angewandte Mathematik, Westf\"alische Wilhelms-Universit\"at (WWU) M\"unster. Einsteinstr. 62, D 48149 M\"unster, Germany. e-mail: {\tt jan.pietschmann@wwu.de}  }
\author[M. Schlottbom]{Matthias Schlottbom$^\diamond$}
\thanks{$^\diamond$ {D}epartment of {A}pplied {M}athematics, {U}niversity of {T}wente,
P.O. Box 217, 7500 AE Enschede, The Netherlands.
E-mail: {\tt m.schlottbom@utwente.nl}
}

\begin{abstract}
We consider the identification of nonlinear diffusion coefficients of the form $a(t,u)$ or $a(u)$ in quasi-linear parabolic and elliptic equations. Uniqueness for this inverse problem is established under very general assumptions using partial knowledge of the Dirichlet-to-Neumann map.
The proof of our main result relies on the construction of a series of appropriate Dirichlet data and test functions with a particular singular behavior at the boundary. 
This allows us to localize the analysis and to separate the principal part of the equation from the remaining terms. 
We therefore do not require specific knowledge of lower order terms or initial data which allows 
to apply our results to a variety of applications. This is illustrated by discussing some typical examples in detail.
\end{abstract}


\maketitle

\section{Introduction}
Consider the inverse problem of identifying $a=a(t,u)$ for $0 \le t \le T$ and $\underline g \le u \le \overline g$ in a second order quasi-linear parabolic differential equation of the form
\begin{align} \label{eq:par}
   - \div ( a(t,u) \nabla u + b(x,t,u)) + c(x,t,u,\nabla u) &= \ddt d(t,u) .
\end{align}
The equation shall hold on a bounded domain $\Omega \subset \RR^d$, $d=2,3$ with piecewise smooth boundary $\partial \Omega$ and for all $0<t<T$. 
We assume that for any choice of boundary values
\begin{align}
  u &= g \qquad\text{on } \partial\Omega \times (0,T),\label{eq:bc}
\end{align}
in an appropriate class $G$ of Dirichlet data 
and for coefficient functions $a$, $b$, $c$, $d$ satisfying some structural properties,
there exists a bounded weak solution with initial values 
\begin{align} \label{eq:ic} 
u(0)=u_0 \qquad \text{in } \Omega. 
\end{align}
Further, we assume to have access to additional measurements of the boundary fluxes 
\begin{align}\label{eq:flux}
 j = n \cdot (a(t,u) \nabla u + b(x,t,u)) \qquad\text{on } \Gamma_M\times (0,T),
\end{align}
on some non-trivial smooth part $\Gamma_M \subset \partial \Omega$ of the boundary.
As usual, $n$ here denotes the outward unit normal vector on $\partial\Omega$.
Our main result for this parabolic problem is that partial information $\{(g,j(g))\}_{g \in G}$ of the Dirichlet-to-Neumann map uniquely determines the diffusion coefficient $a(t,u)$ for $0 \le t \le T$ and $\underline g \le u  \le \overline g$. More precisely, 
\begin{align} \label{eq:result1}
\text{if } a_1 \not\equiv a_2 \text{ on } (0,T) \times (\underline g,\overline g),
\quad \text{then there exists } g \in G: j_1 \not\equiv j_2 \text{ on } \Gamma_M \times (0,T).
\end{align}
By $j_i$, $i=1,2$, we denote the boundary fluxes \eqref{eq:flux} 
for solutions $u_i$ of \eqref{eq:par}--\eqref{eq:ic} with parameter functions 
$a_i$, $b_i$, $c_i$, $d_i$, and initial values $u_{i,0}$, respectively. 
No detailed knowledge of the lower order terms or the initial values will be required for the proof of the above assertion. 
The reverse statement that $a_1 \equiv a_2$ implies $j_1 \equiv j_2$ 
of course needs additional assumptions. This fact and the identification of $a(u)$ in the corresponding elliptic problem will also be discussed in detail later.

Uniqueness of unknown parameters in partial differential equations is one of the main research topics in the field of inverse problems. The most prominent example probably is Calder{\'o}n's problem \cite{Calderon80}, where one aims to reconstruct an unknown spatially varying conductivity $a=a(x)$ in the elliptic equation $\div(a(x) \nabla u) = 0$ from observation of the full or partial Dirichlet-to-Neumann map; we refer to \cite{Isakov93b,KenigSalo2014,Uhlmann2009} for comprehensive reviews. 

The identification of $a=a(u)$ in the quasilinear elliptic equation $-\div (a(u) \nabla u)=0$ has been investigated by Cannon \cite{Cannon1967}, who gave a constructive proof for the determination of the coefficient from knowledge of a single measurement of $u$ along a curve on $\partial \Omega$. A stable numerical method for the problem has been proposed in \cite{EPS2014b}.
%
Simultaneous identification of two parameters $a$ and $c$ in the elliptic equation $-\div(a(u) \nabla u) + c(x) u=0$ has been considered in \cite{EPS2014a}.
A survey of further related results can be found in \cite[Ch.~5]{Isakov06}.

The identification of parameters in nonlinear parabolic problems from boundary measurements was considered in Cannon and DuChateau \cite{Cannon1973} which initiated a series of papers treating the spatially one-dimensional case \cite{Cannon1980,Cannon1989,DuChateau1981,DuChateau2004}. These results are based on monotonicity properties of the solution and use {\em adjoint methods} to prove uniqueness of the inverse problem with overdetermined boundary data; see also \cite{Cortazar1990} where the nonlinearity is allowed to degenerate.
A special case is treated by Lorenzi \cite{Lorenzi1986}, who assumes that $a(u)$ is already known on some interval which lies in the range of the initial datum.
The identification of nonlinear lower order terms $c=c(u)$, $c=c(x,u)$, or $c=c(u,\nabla u)$ in parabolic problems has been investigated in \cite{DuChateauRundell85,Isakov93,Isakov01}.
For an extensive overview of available results and further references on parameter identification in partial differential equations from single and multiple boundary measurements, see e.g. \cite{Isakov06,KlibanovTimonov04,Yamamoto09}.

The proof of the main result of this manuscript is based on the following rationale:
We start from a variational formulation of the problem~\eqref{eq:par}--\eqref{eq:bc}.
Due to its special form, the principal part can be reduced to a boundary integral 
if harmonic test functions are used in this variational principle.
%
We then construct a sequence $\varphi^\eps$ of harmonic test functions with a particular singular behavior in the limit $\eps\to 0$ and choose appropriate Dirichlet boundary data $g^\eps$ which vary locally around points where the right-hand side of \eqref{eq:result1} holds true. 
When inserting the corresponding solutions and test functions into the variational principle, 
one can see that the lower order terms scale differently with respect to $\eps$ compared to the principal part. The latter term however is localized at the boundary and can be fully controlled,
which allows us to prove the validity of \eqref{eq:result1} by contradiction.
The main arguments used in our proofs are rather general and allow us to extend 
the above results in various directions. Some particular results and examples for 
applications will be presented at the end of the manuscript.

Let us mention that singular functions have already been used successfully for uniqueness and stability proofs before, e.g., by Alessandrini~ \cite{Alessandrini90} in the context of the Calder\'on problem and by von~Harrach \cite{Harrach09}, who constructed particular functions via the unique continuation principle, to obtain identifiability results for the Calder\'on problem and a related problem with additional lower order term $c=c(x)u$.
\medskip

This paper is organized as follows: 
In Section~\ref{sec:ass}, we first introduce our basic assumptions and then 
formulate in detail our main result concerning the parabolic problem discussed above. 
In Section~\ref{sec:auxiliary}, we construct the singular test functions and derive some auxiliary estimates and in Section~\ref{sec:proof}, we complete the proof our main assertion. 
In Section~\ref{sec:additional}, we establish the converse implication and state the corresponding results for the elliptic problem. In addition, we discuss possible extensions of our results concerning our assumptions on the geometry, the regularity of parameters, or the boundary conditions.
To illustrate the applicability of our approach, we discuss in Section~\ref{sec:examples} 
possible applications, including parabolic problems in bioheat transfer, systems of coupled equations arising in chemotaxis and urban crime modeling, as well as  systems describing electron migration in semi-conductors and nanopores.

\section{Assumptions and main result} \label{sec:ass}

Let us start by introducing a few general assumptions that will be utilized for our analysis throughout the text. 
The first assumption concerns the geometric setting.
\begin{assumption}\label{ass:1}
  $\Omega\subset\RR^d$, $d=2,3$ is a bounded Lipschitz domain and 
  $\Gamma_M' \subset \partial\Omega$ is an open smooth part of the boundary with $|\Gamma_M'| \ne 0$, 
  i.e., there exists $\bar x \in \Gamma_M'$ and $\eps_0>0$ such that $B_{\eps_0}(\bar x + \eps_0 n(\bar x)) \cap \Omega = \emptyset$ and $\Gamma_M = B_{\eps_0}(\bar x) \cap \partial\Omega \subset \Gamma_M'$. 
  We further assume that $\Gamma_M$ is flat, i.e., $n(x) = n(\bar x)=e_d$ on $\Gamma_M$ with $e_d$ being the $d$-th unit vector; see Figure~\ref{fig:geometry} for a sketch. 
\end{assumption}
The last assumption is made for convenience of notation and can be relaxed to $\Gamma_M$ being of class $C^1$ by the usual localization argument.
With the second assumption, we introduce some general conditions on the parameter functions that appear in the parabolic problem under consideration.
\begin{assumption}\label{ass:2}
The parameters $a,b,c,d$ lie in $W^{1,\infty}$ with norm bounded by $C_A>0$ and $0 < \underline a \le a(t,u) \le \overline a$ for some $\underline a, \overline a>0$. Moreover, $u_0 \in L^2(\Omega)$ with $\|u_0\|_{L^2(\Omega)} \le C_0$.
\end{assumption}
Again, some of these assumptions can be relaxed considerably, which will become clear from the proofs.
Since the parabolic problem under investigation is rather general, we assume for the moment the existence of solutions and uniform bounds. The validity of these mild assumptions has of course to be verified when considering particular applications. 
\begin{assumption} \label{ass:3}
Let Assumption~\ref{ass:1} and \ref{ass:2} hold. Then for any Dirichlet datum 
$$
g \in  G:=\{g \in H^1(0,T;H^1(\partial\Omega)) : \underline g \le g \le \overline g \quad \text{and} \quad \|g\|_{H^1(0,T;H^1(\partial\Omega))} \le C_G\}
$$
with parameters $\underline g$, $\overline g$, and constant $C_G$, 
there exists $u \in L^2(0,T;H^1(\Omega)) \cap C([0,T];L^2(\Omega))$ satisfying \eqref{eq:bc} and \eqref{eq:ic} in the sense of traces and \eqref{eq:par} in the sense of distributions, i.e., 
\begin{align}\label{eq:weak}
  \int_0^T (a(t,u)\nabla u + b(x,t,u), \nabla \varphi)_\Omega + ( c(x,t,u,\nabla u), \varphi)_\Omega dt = -\int_0^T (d(t,u), \partial_t \varphi)_\Omega dt
\end{align}
for all 
$\phi \in H_0^1(0,T;H^1_0(\Omega))$; 
as usual, $(u,v)_{\Omega}$ denotes the scalar product of $L^2(\Omega)$ here. 
Moreover, any such weak solution of \eqref{eq:par}--\eqref{eq:bc} is bounded uniformly by
\begin{align}\label{eq:apriori}
  \|u\|_{L^2(0,T;H^1(\Omega))} \leq C_U
\end{align}
with a constant $C_U=C_U(\underline g,\overline g,C_A,C_G,C_0,\Omega)$ that is independent of the particular choice of the coefficients, of the initial value, and of the Dirichlet datum $g \in G$. 
\end{assumption}
Here and below we use standard notation for function spaces,  
in particular, $H^{-1}(\Omega) = H_0^1(\Omega)'$ is the dual space of $H_0^1(\Omega)$, and $H^{s}(0,T;X)$ denotes the appropriate Bochner space of functions $f:(0,T) \to X$ with values in some Banach-space $X$; let us refer to \cite{Evans98} for details.
The above assumptions on the parameters allow us to define for any bounded weak solution $u \in L^2(0,T;H^1(\Omega))$ the Neumann flux or generalized co-normal derivative
$j=n \cdot (a(t,u) \nabla u + b(x,t,u))$ as a linear functional on $H_0^1(0,T;H^1(\Omega))$ via
\begin{align}\label{eq:conormal}
\int_0^T\langle j,\varphi\rangle_{\partial \Omega} dt
:=  \int_0^T (a(t,u)\nabla u &+ b(x,t,u), \nabla\varphi)_\Omega \\
&+ (c(x,t,u,\nabla u),\varphi)_\Omega + (d(t,u),\dt \varphi)_\Omega dt,\nonumber
\end{align}
for all $\varphi\in  H^1_0(0,T;H^1(\Omega))$.
Using Assumptions~\ref{ass:1}--\ref{ass:3}, one can directly see that 
\begin{align} \label{eq:flux_estimate}
  &\int_0^T\langle j,\varphi\rangle_{\partial \Omega} dt
  \leq  C_A (3+C_U) \|\varphi\|_{H^1(0,T;H^1(\Omega))}.
\end{align}
which establishes a uniform bound for $j$ in the norm of linear functionals on $H_0^1(0,T;H^1(\Omega))$.
Using standard convention, we say that $j \equiv 0$ on $\Gamma_M\times (0,T)$, 
if 
\begin{align*}
	  &\int_0^T\langle j,\varphi\rangle_{\partial \Omega} dt=0
\end{align*}
holds for all $\varphi\in H^1_0(0,T;H^1(\Omega))$ with $\varphi \equiv 0$ on $ \partial\Omega \setminus \Gamma\times (0,T)$ in the sense of traces. 
Accordingly, we have $j_1\equiv j_2$ on $\Gamma_M\times (0,T)$ if $j_1-j_2\equiv0$ on $\Gamma_M\times (0,T)$.

\medskip 

We are now in the position to state our main result in a rigorous manner.
\begin{theorem} \label{thm:1}
Let Assumption~\ref{ass:1} hold and let $a_i$, $b_i$, $c_i$, $d_i$, and $u_{0,i}$, for $i=1,2$
satisfy Assumption~\ref{ass:2}. For $g \in G$, let $u_i(g)$
be corresponding weak solutions of \eqref{eq:par}--\eqref{eq:ic} in the sense of Assumption~\ref{ass:3} and let $j_i(g)$ denote the corresponding Neumann fluxes.
Assume that $a_1(\tilde t,\tilde g) \neq a_2(\tilde t,\tilde g)$ for some $0 < \tilde t < T$ and $\underline g \le \tilde g \le \overline g$, then there exists a Dirichlet datum $g \in G$ such that $j_1(g) \not \equiv j_2(g)$ on $\Gamma_M\times (0,T)$.
\end{theorem}
Some extensions of this result will be stated in Section~\ref{sec:additional}. 
In the proof of Theorem~\ref{thm:1}, which is presented in Section~\ref{sec:proof}, 
we will use particular test functions $\varphi$ in the definition of the co-normal derivative \eqref{eq:conormal} with a very specific singular behavior at the boundary. 
The next section presents some auxiliary results required for the construction of these functions.
\section{Auxiliary results} \label{sec:auxiliary}

Let $\Phi$ denote the fundamental solution of the Laplace equation, i.e., 
\begin{align*}
\Phi(x) = \begin{cases} -\frac{1}{2\pi} \log |x|, & d = 2, \\
\frac{1}{4\pi} \frac{1}{|x|}, & d=3. \end{cases}
\end{align*}
For $\bar{x} \in \partial\Omega$ and $\eps_0>0$ as defined Assumption~\ref{ass:1}, 
and for any $0 < \eps \le \eps_0$, we define 
\begin{align}\label{eq:lambda}
 \lambda^\eps_{\bar x}(x) = n(\bar x)\cdot \nabla \Phi(x-\bar x^\eps), \qquad \bar x^\eps = \bar x + \eps n(\bar x)
\end{align}
for all $x \ne \bar x + \eps n(\bar x)$, and we set $\lambda^\eps_{\bar x}(\bar x^\eps) = 0$ for completeness. 
%
By construction and Assumption~\ref{ass:1}, we have $\bar x^\eps \notin \Omega$ for $0 < \eps \le \eps_0$, and hence $\lambda^\eps_{\bar x}$ is a smooth function in $\Omega$. 
\begin{lemma}\label{lem:fundamental}
For every $0 < \eps \le \eps_0$, we have $\lambda_{\bar x}^\eps \in C^\infty(\overline \Omega)$ and $\Delta \lambda_{\bar x}^\eps(x) = 0$ for all $x \in \Omega$. 
In addition, there exists a constant $C_L=C_L(p,\Omega,\eps_0) > 0$ independent of $\eps$ such that
 \begin{align*}
   \|\lambda^\eps_{\bar x}\|_{L^p(\Omega)} \leq  C_L
    \begin{cases} 
	1,                            & p<d/(d-1),\\ 
	|\ln(\eps)|^{\frac{1}{p}},    & p=d/(d-1),\\
	\eps^{1+\frac{d}{p}-d},       & p>d/(d-1).
    \end{cases}
 \end{align*}
The gradient of $\lambda_{\bar x}^\eps$ can further be estimated by 
 \begin{align*}
   \|\nabla \lambda^\eps_{\bar x}\|_{L^p(\Omega)} \leq  C_L
    \begin{cases} 
	|\ln(\eps)|,          & p=1,\\
	\eps^{\frac{d}{p}-d}, & p>1.
    \end{cases}
 \end{align*}
\end{lemma}
\begin{proof}
Smoothness of $\lambda_{\bar x}^\eps$ and $\Delta \lambda_{\bar x}(x)=0$ follow by direct computation. 
By Assumption~\ref{ass:1}, we further know that $\Omega \subset B_R(\bar x^\eps) \setminus B_{\eps}(\bar x^\eps)$ for $0 < \eps \le \eps_0$; see~Figure~\ref{fig:geometry}. Hence 
\begin{align*}
\int_{\Omega} |\lambda_{\bar x}^\eps|^p dx 
&\le C'\int_{B_R(\bar x^\eps) \setminus B_\eps(\bar x^\eps)} |x-\bar x^\eps|^{(1-d)p} dx  
\le C'' \int_{\eps}^R r^{(1-d)(p-1)} dr.
\end{align*}
The estimates for the norm of $\lambda_{\bar x}^\eps$ then follow directly by integration. 
The gradient of $\lambda_{\bar x}^\eps$, on the other hand, can be estimated by $|\nabla \lambda_{\bar x}^\eps(x)| \le C' |x- \bar x^\eps|^{-d}$ and hence
\begin{align*}
\int_{\Omega} \|\nabla \lambda_{\bar x}^\eps\|^p dx 
&\le C''\int_{B_R(\bar x^\eps) \setminus B_\eps(\bar x^\eps)} |x-\bar x^\eps|^{-dp} dx  
\le C''' \int_{\eps}^R r^{d(1-p)-1} dr.
\end{align*}
The estimates for $\nabla \lambda_{\bar x}^\eps$ then again follow directly by computing this integral.
\end{proof}

\begin{figure}[ht]\label{fig:geometry}
  \providecommand\rotatebox[2]{#2}%
  \ifx\svgwidth\undefined%
    \setlength{\unitlength}{300bp}
    \ifx\svgscale\undefined%
      \relax%
    \else%
      \setlength{\unitlength}{\unitlength * \real{\svgscale}}%
    \fi%
  \else%
    \setlength{\unitlength}{\svgwidth}%
  \fi%
  \global\let\svgwidth\undefined%
  \global\let\svgscale\undefined%
  \makeatother%
  \begin{picture}(1,0.73868265)%
    \put(0,0){\includegraphics[width=\unitlength]{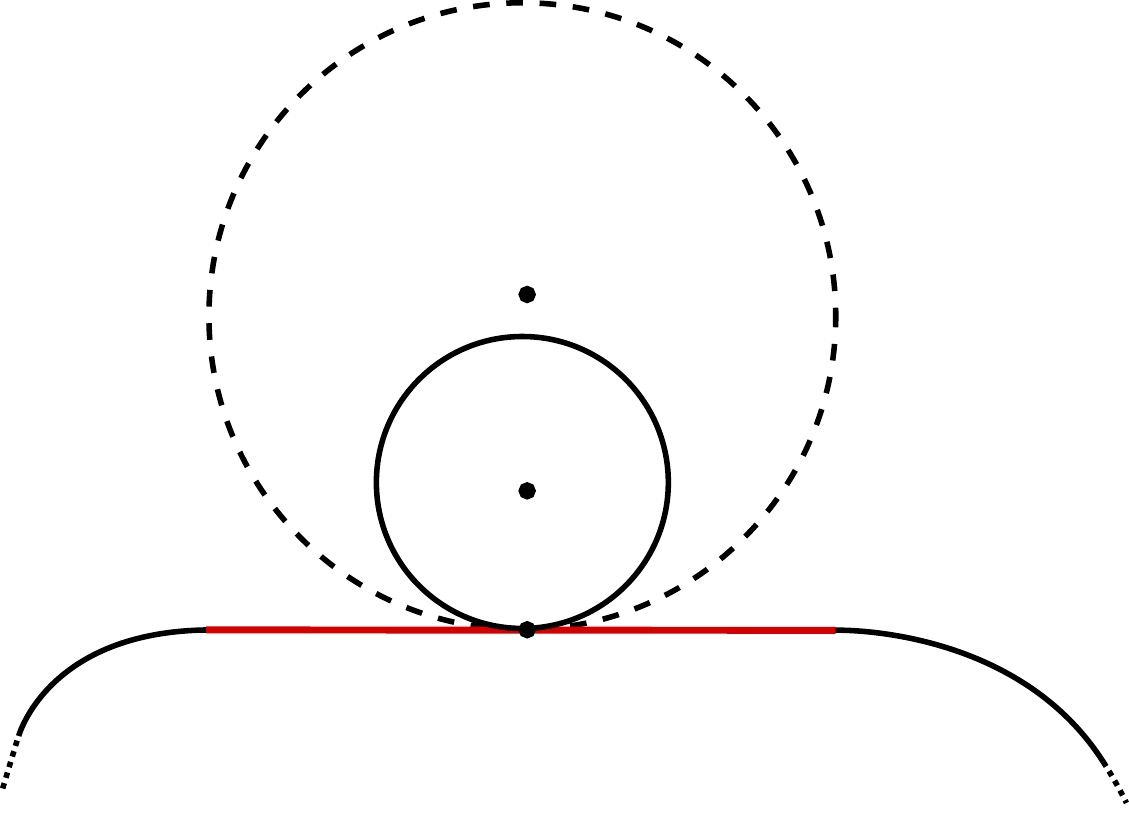}}%
    \put(0.25858704,0.41522485){$B_\eps(\bar x^\eps)$}%
    \put(0.42530482,0.48632676){$\bar x^{\eps_0}$}%
    \put(0.74211317,0.53444393){$B_{\eps_0}(\bar x^{\eps_0})$}%
    \put(0.4243913,0.30692857){$\bar x^\eps$}%
    \put(0.45127747,0.14030053){$\bar x$}%
    \put(0.6193288,0.13015073){$\Gamma_M$}%
    \put(0.088181,0.0663975){\Large$\Omega$}%
    \put(0,0){\includegraphics[width=\unitlength,page=2]{sketch_bdry.pdf}}%
    \put(0.25001361,0.1405804){$x$}%
  \end{picture}%
\caption{Sketch of the geometry near $\Gamma_M$ and the singular points $\bar x^\eps=\bar x + \eps n(\bar x)$.}
\end{figure} 

The next result describes in more detail the behavior of $\lambda^\eps_{\bar x}$ away from the singularity.
\begin{lemma}\label{lem:bound_on_lambda}
 For any $0<\eps \le \eps_0$ and all $x \in \overline\Omega$ with $|x-\bar x| \ge \eps_0/2$ 
there holds 
 \begin{align*}
  | \lambda_{\bar x}^\eps (x)| \le C_L' \qquad \text{and} \qquad |\nabla \lambda_{\bar x}^\eps (x)| \leq C_L'
 \end{align*}
with constant $C_L'=C_L'(\Omega,\eps_0)$ independent of $\eps$ and the choice of $x$.
\end{lemma}
\begin{proof}
First consider the case that $\eps \le \eps_0/4$. Then
$$
|x-\bar x^\eps| \ge |x-\bar x| - |\bar x - \bar x^\eps| \ge \eps_0/2 - \eps_0/4 = \eps_0/4.
$$
Otherwise, we have $\eps_0/4 \le \eps \le \eps_0$, and thus
$$
|x-\bar x^\eps| \ge |x-\bar x^{\eps_0}| - |\bar x^{\eps_0} - \bar x^\eps|
\ge \eps_0 - (\eps_0-\eps_0/4) = \eps_0/4.
$$
Hence $|x-\bar x^\eps| \ge \eps_0/4$ in both cases.
Using the particular form of $\lambda_{\bar x}^\eps$, we therefore obtain  the bounds $|\lambda_{\bar x}^\eps| \le C |x-\bar x^\eps|^{1-d} \le C' \eps_0^{1-d}$ and $|\nabla \lambda_{\bar x}^\eps(x)| \le C |x-\bar x^\eps|^{-d} \le C' \eps_0^{-d}$.
\end{proof}

As a next step, let us also characterize in more detail the behavior of 
$\lambda_{\bar x}^\eps$ in the neighborhood of the point $\bar x\in \Gamma_M$ on the boundary.
\begin{lemma}\label{lem:estimate_dn_lambda}
There exists a constant $C_L''>0$, such that for any $0<\eps<\eps_0$  
$$
  \int_{\Gamma_M \cap B_{\eps}(\bar x)} \partial_n \lambda_{\bar x}^\eps ds(x) \geq C_L''/\eps
\qquad \text{and} \qquad 
\partial_n \lambda_{\bar x}^\eps \geq 0 \quad \text{on } \Gamma_M \cap B_{\eps}(\bar x).
$$

\end{lemma}
\begin{proof}
We only consider the case $d=3$ in detail.
By Assumption~\ref{ass:1}, $n(x)=n(\bar x)$ on $\Gamma_M$. A brief inspection of Figure~\ref{fig:geometry} further reveals that $|x-\bar x^\eps|^2=|x-\bar x|^2 + \eps^2$ for $x \in \Gamma_M$ and $n(\bar x) \cdot (x-\bar x^\eps) = -\eps$. We therefore obtain 
\begin{align*}
4 \pi \partial_n \lambda_{\bar x}^\eps (x) 
&= 3 \frac{|n(\bar x) \cdot (x-\bar x^\eps)|^2}{|x-\bar x^\eps|^5} - \frac{|n(\bar x)|^2}{|x-\bar x^\eps|^3} 
= \frac{3 \eps^2 - |x-\bar x|^2 -\eps^2}{|x-\bar x^\eps|^5}. 
\end{align*}
This shows $\partial_n \lambda_{\bar x}^\eps \ge 2^{-5/2} \eps^{-3}$ for all $x \in \Gamma_M \cap B_{\eps}(\bar x)$ and thus positivity of the normal derivative. The lower bound on the integral follows by noting that $|\Gamma_M \cap B_\eps(\bar x)| = \pi \eps^2 /4$. 
The estimate for dimension $d=2$ can be derived in a similar way. 
\end{proof}
\section{Proof of the Theorem~\ref{thm:1}} \label{sec:proof}
We can now turn to the proof of our main result which proceeds by contradiction.
\subsection{Basic assumptions and integral identity}
Let the assumptions of Theorem~\ref{thm:1} hold but assume 
that there exists some $(\tilde t,\tilde g) \in (0,T) \times (\underline g,\overline g)$ 
such that $a_1(\tilde t, \tilde g) \ge a_2(\tilde t, \tilde g) + 2\eta$ with $\eta>0$.
By Assumption~\ref{ass:2}, we know that $a_1$ and $a_2$ are continuous and hence 
\begin{align} \label{eq:ass}
a_1(t,g) \ge a_2(t,g) + \eta, \qquad \text{for all } t \in (t_1,t_2) \text{ and } g \in (g_1,g_2)
\end{align}
for some appropriate intervals $(t_1,t_2)$ and $(g_1,g_2)$ around $\tilde t$ and $\tilde g$.
Let us denote by 
\begin{align} \label{eq:prim}
A_i(t,g) = \int_{g_1}^g a_i(t,u) du, \qquad i=1,2 
\end{align}
the anti-derivatives of the diffusion parameters $a_i(t,g)$. 
The following identity will be one of the central arguments for the proof of Theorem~\ref{thm:1}.
\begin{lemma} \label{lem:identity}
Let the assumptions of Theorem~\ref{thm:1} be valid. Then 
\begin{align*} 
&\int_0^T ( A_1(t,u_1) - A_2(t,u_2), \partial_n \varphi )_{\partial\Omega} dt
=  \int_0^T  \langle j_1-j_2,\varphi\rangle_{\partial \Omega} -\left(d_1(t,u_1) - d_2(t,u_2), \partial_t\varphi\right)_\Omega\\
& \qquad  -\left(b_1(x,t,u_1)-b_2(x,t,u_2),\nabla\varphi\right)_\Omega - \left( c_1(x,t,u_1,\nabla u_1) -c_2(x,t,u_2,\nabla u_2), \varphi \right)_\Omega dt\notag
\end{align*}
for all $\varphi\in H_0^1(0,T;H^1(\Omega))$ with $\Delta \varphi=0$ on $\Omega \times (0,T)$.
\end{lemma}
\begin{proof}
Subtracting the two equations \eqref{eq:conormal} which define the fluxes $j_i$, we obtain  
\begin{align*}
\int_0^T \langle j_1-j_2,\varphi\rangle_{\partial \Omega}dt 
&= \int_0^T \left( a_1(t,u_1)\nabla u_1 - a_2(t,u_2)\nabla u_2, \nabla \varphi \right)_\Omega 
  + \left(d(t,u_1) - d(t,u_2), \dt \varphi\right)_\Omega \\
& \qquad + \left( b_1(x,t,u_1) - b_2(x,t,u_2), \nabla \varphi \right)_\Omega \\
& \qquad + \left( c_1(x,t,u_1,\nabla u_1) - c_2(x,t,u_2,\nabla u_2), \varphi \right)_\Omega  dt
\end{align*}
for all $\varphi\in H_0^1(0,T;H^1(\Omega))$. We can now express the terms $a_i(t,u_i) \nabla u_i = \nabla A_i(t,u_i)$ via the primitives, and then use integration-by-parts for the first term on the right hand side as well as $\Delta\varphi=0$ to obtain the result.
\end{proof}

\subsection{Construction of Dirichlet data}
We now define spatial and temporal cutoff functions that will be used for localization of the following estimates in space and time, i.e.,
\begin{align} \label{eq:psieps}
\chi_{\bar x}^\eps(x) &=  \begin{cases} 1, & |x-\bar x| \le \eps/2, \\ 2-2 |x-\bar x|/\eps, & \eps/2 < |x-\bar x| < \eps,\\ 0, & \text{else}, \end{cases}
\intertext{and}
\chi(t) &= \max\{(t-t_1)(t_2-t),0\}. 
\end{align}
With the help of these auxiliary functions, we can now construct candidates for appropriate Dirichlet data $g$ to be used for the proof of Theorem~\ref{thm:1}. 
This construction also yields a constant $C_G$ to be used in Assumption~\ref{ass:3}.
\begin{lemma} \label{lem:dirichlet}
There exist positive constants $\gamma$ and $C_G$ such that for any $0<\eps \le \eps_0$
the function $g_{\bar x}^\eps(x,t) = g_1 + \gamma \eps^{(3-d)/2}\chi_{\bar x}^\eps(x) \chi(t)$ 
satisfies 
\begin{align*} 
\underline g \leq g_1\le g_{\bar x}^\eps \le g_2\leq \bar g
\qquad \text{and} \qquad 
\|g_{\bar x}^\eps\|_{H^1(0,T;H^1(\partial\Omega))} \le C_G.
\end{align*}
\end{lemma} 
\begin{proof}
Since $\chi$ is piecewise smooth, $0\leq \chi(t)\leq T^2/4$ and $|\chi'(t)|\leq T$, the functions $g^\eps$ are uniformly bounded and also differentiable with respect to $t$ and by choosing $\gamma=4(g_2-g_1)\min(1,\eps_0^{(d-3)/2})/T^2$,
we can satisfy the asserted pointwise bounds. 
We therefore only have to consider the spatial derivatives in detail. By direct computation
\begin{align*}
\|\nabla g_{\bar x}^\eps(\cdot,t)\|_{L^2(\partial\Omega)} \le  \frac{\gamma T^2}{4} \eps^{(3-d)/2}\|\nabla \chi_{\bar x}^\eps\|_{L^2(\Gamma_M \cap B_{\eps}(\bar x))}.
\end{align*}
 Furthermore, by definition of $\chi_{\bar x}^\eps$,
\begin{align*}
\|\nabla \chi_{\bar x}^\eps\|_{L^2(\Gamma_M \cap B_{\eps}(\bar x))}^2 
&\le C' \int_{\eps/2}^{\eps} \eps^{-2} r^{d-2} dr = C'' \eps^{d-3}. 
\end{align*}
This shows that $\|\nabla g^\eps_{\bar x}(\cdot,t)\|_{L^2(\partial\Omega)} \le \gamma C_1$, and 
the same estimate can be obtained for the time derivative $\|\dt \nabla g^\eps_{\bar x}(\cdot,t)\|_{L^2(\partial\Omega)}$.
Thus we obtain 
\begin{align*}
\|g_{\bar x}^\eps\|_{H^1(0,T;H^1(\partial\Omega))} \le 2 \gamma C_1 + C_2, 
\end{align*}
and we may choose $C_G=2 \gamma_0 C_1 + C_2$ with $\gamma_0=4(\bar g-\underline g) \min(1,\eps_0^{(d-3)/2})/T^2\geq \gamma$ as uniform bound for the norm.
\end{proof}

By Assumption~\ref{ass:3} we therefore know, that for any $g=g_{\bar x}^\eps$ as defined above, 
we have weak solutions $u_i(g)$ and corresponding fluxes $j_i(g)$ as required in Theorem~\ref{thm:1}. 

\subsection{Test function and estimates for the flux}
We test the identity of Lemma~\ref{lem:identity} with appropriate test functions $\varphi$ to prove the theorem. 
\begin{lemma} \label{lem:estimates}
Set $\varphi_{\bar x}^\eps(x,t) =  \lambda_{\bar x}^\eps(x)\chi(t)$ and let $u_i^\eps=u_i(g_{\bar x}^\eps)$ be weak solutions of \eqref{eq:par}--\eqref{eq:ic} in the sense of Assumption~\ref{ass:3} for Dirichlet data $g=g_{\bar x}^\eps$ and with parameter functions $a=a_i$, $b=b_i$, $c=c_i$, $d=d_i$, and initial values $u_{i,0}$ satisfying Assumption~\ref{ass:2}. Then 
\begin{align*}
\int_0^T ( A_1(t,u_1^\eps) - A_2(t,u_2^\eps), \partial_n \varphi_{\bar x}^\eps)_{\partial\Omega}\d t \ge C_1 \eps^{(1-d)/2}. 
\end{align*}
If, in addition, $j_1 \equiv j_2$ on $\Gamma_M\times (0,T)$, then
\begin{align*}
\Big|\int_0^T  \langle j_1-j_2,\varphi_{\bar x}^\eps\rangle_{\partial \Omega} -\left(d_1(t,u_1) - d_2(t,u_2), \partial_t\varphi_{\bar x}^\eps\right)_\Omega -\left(b_1(x,t,u_1)-b_2(x,t,u_2),\nabla\varphi_{\bar x}^\eps\right)_\Omega \\
- \left( c_1(x,t,u_1,\nabla u_1) - c_2(x,t,u_2,\nabla u_2), \varphi_{\bar x}^\eps \right)_\Omega dt \Big|\notag
 \le C_2  |\ln(\eps)|
\end{align*}
for all $0 < \eps \le \eps_0$ with some positive constants $C_1$, $C_2$ independent of $\eps$. 
\end{lemma}
\begin{proof}
Using the particular construction of the anti-derivatives $A_i$, of the Dirichlet datum $g_{\bar x}^\eps$, and of the test function $\varphi_{\bar x}^\eps$, we directly  obtain
\begin{align*}
( A_1(t,u_1^\eps) - A_2(t,u_2^\eps), \partial_n \varphi_{\bar x}^\eps(t))_{\partial\Omega}
=  \big(A_1(t,g_{\bar x}^\eps) - A_2(t,g_{\bar x}^\eps),  \chi(t)\partial_n \lambda^\eps_{\bar x}  \big)_{\Gamma_M\cap B_{\eps}(\bar x)}.
\end{align*}
Since $g_1\leq g_{\bar x}^\eps\leq g_2$ on $B_{\eps}(\bar x)\cap \Gamma_M\times (t_1,t_2)$ by Lemma~\ref{lem:dirichlet} and $a_1-a_2 \ge \eta$ on $(t_1,t_2)\times(g_1,g_2)$ by assumption \eqref{eq:ass}, we can use Lemma~\ref{lem:estimate_dn_lambda} and some elementary computations 
to see that
\begin{align*}
  \int_{t_1}^{t_2} \big( A_1(t,g^\eps_{\bar x}) - A_2(t,g^\eps_{\bar x}), \chi(t) \partial_n \lambda^\eps_{\bar x} \big)_{\Gamma_M\cap B_\eps(\bar x)} \d t
 \geq C_1\eps^{\frac{1-d}{2}}.
\end{align*}
Since $j_1 \equiv j_2$ on $\Gamma_M\times (t_1,t_2)$ and $\chi^{\eps_0}_{\bar x} \equiv 0$ on $\partial\Omega \setminus \Gamma_M$, we further see that 
\begin{align*}
	\int_0^T \langle j_1-j_2,\chi(t) \lambda_{\bar x}^\eps \rangle_{\partial\Omega} dt &= \int_0^T \langle j_1-j_2,\chi(t)(1-\chi^{\eps_0}_{\bar x})\lambda_{\bar x}^\eps \rangle_{\partial\Omega} dt\\ 
	&\leq  C_A (3+C_U) \| (1-\chi^{\eps_0}_{\bar x}) \lambda^\eps_{\bar x} \chi\|_{H^1(0,T;H^1(\Omega))} \leq C,
\end{align*}
where we used the estimate \eqref{eq:flux_estimate} for the flux and the uniform bounds for the cutoff functions and $\lambda_{\bar x}^\eps$ provided by Lemma~\ref{lem:bound_on_lambda}.
The constant $C$ therefore can be chosen independent of $\eps$. 
Using Lemma~\ref{lem:fundamental}, we can further estimate 
\begin{align*}
	\int_0^T\left(b_1(x,t,u_1)-b_2(x,t,u_2),\nabla\varphi_{\bar x}^\eps\right)_\Omega dt \leq C |\ln(\eps)|
\end{align*}
and the other terms of lower order can be estimated similarly due to the bounds of Assumption~\ref{ass:3} and the uniform estimates for the singular function $\lambda_{\bar x}^\eps$.
\end{proof}

\subsection{Proof of Theorem~\ref{thm:1}}
For $\eps$ sufficiently small, the estimates of Lemma~\ref{lem:estimates} are in contradiction to the identity of Lemma~\ref{lem:identity}. 
Thus the assumption \eqref{eq:ass} cannot be valid. This concludes the proof of Theorem~\ref{thm:1}. \qed


\section{Further results and extensions} \label{sec:additional}

We now present some further results, that can be deduced or derived in a similar way as Theorem~1,
and then discuss some possibilities for relaxing the assumptions.

\subsection{The reverse implication} 

As mentioned in the introduction, the reverse result
\begin{align}
a_1 \equiv a_2 \text{ on } (0,T) \times (\underline g,\overline g) 
\implies 
j_1(g) \equiv j_2(g) \text{ on } \Gamma_M\times (0,T) \text{ for all } g \in G
\end{align}
does in general not hold, unless more information about the lower order terms and the solution 
is available. A mismatch of the Neumann data therefore does not allow to deduce a difference in the parameter. 
For the following assertion, we additionally require 
\begin{assumption} \label{ass:4}
For any $g \in G$ and initial value $u_0$ satisfying Assumption~\ref{ass:2}, the weak solution $u(g)$ of \eqref{eq:par}--\eqref{eq:ic} specified in Assumption~\ref{ass:3} is unique and satisfies $\underline g \le u(g) \le \overline g$.
\end{assumption}
In many applications, Assumption~\ref{ass:4} can be verified by comparison principles or similar considerations. From the results of the previous sections and the above considerations, 
we can now directly deduce validity of the following assertion. 
\begin{theorem} \label{thm:2}
Let the assumptions of Theorem~\ref{thm:1} hold. Furthermore, let Assumption~\ref{ass:4} be valid and assume that $b_1=b_2$, $c_1=c_2$, $d_1=d_2$, and $u_{0,1}=u_{0,2}$. Then 
\begin{align*}
a_1 \equiv  a_2 \text{ on } (0,T) \times (\underline g,\overline g) 
\Longleftrightarrow 
j_1(g) \equiv j_2(g)  \text{ on } \Gamma_M\times (0,T) \text{ for all } g \in G. 
\end{align*}
\end{theorem}
Note that the assertion of Theorem~\ref{thm:1} is inline with similar results reported in literature; see, e.g., \cite{Isakov93,Isakov01}. Also there, the validity of the reverse implication provided by Theorem~\ref{thm:2} would require further assumptions. For linear elliptic problems without lower order terms, like the Calder\'on problem, the reverse implication is however trivially satisfied.

\subsection{Elliptic problems}

A brief inspection of the proof of Theorem~\ref{thm:1} shows that similar arguments can be used to derive a corresponding result for the elliptic problem
\begin{align}\label{eq:ell}
 - \div ( a(u) \nabla u + b(x,u)) + c(x,u,\nabla u) &=0 \qquad \text{in } \Omega,\\
  u &= g \qquad\text{on }\partial\Omega.\label{eq:ell_bc1}
\end{align}
The assumptions on the coefficients can now be replaced by 
\begin{assumption} \label{ass:2e} 
The coefficients $a$, $b$, and $c$ lie in $W^{1,\infty}$ with norm bounded by a uniform constant $C_A$ and, in addition, $0 < \underline a \le a \le \overline a$ for some $\underline a,\overline a>0$.
\end{assumption}
Again, we first only have to require existence of weak solutions together with some uniform bounds.
The corresponding assumption for the elliptic case reads
\begin{assumption} \label{ass:3e}
For any $g \in G = \{g \in  H^1(\partial\Omega): \underline g \le g \le \overline g \text{ and } \|g\|_{H^1(\partial\Omega)} \le C_G\}$ there exists a function $u \in H^1(\Omega)$ which satisfies  \eqref{eq:ell_bc1} in the sense of traces and \eqref{eq:ell} in the sense of distributions, i.e., 
\begin{align*}
(a(u) \nabla u + b(x,u), \nabla \phi)_\Omega + (c(x,u,\nabla u),\phi)_\Omega = 0
\end{align*}
for all test function $\phi \in H^1_0(\Omega)$. Moreover, any such weak solution is bounded by
\begin{align*}
\|u\|_{H^1(\Omega)} \le C_U
\end{align*}
with a constant $C_U=C_U(C_A,C_G,\eps_0,\Omega,d,\underline g,\bar g)$ independent of the particular choice of the parameters and the boundary data.
\end{assumption}
The Neumann flux can now be defined as a linear functional on $H^1(\Omega)$ by 
\begin{align}\label{eq:flux_stat}
 \langle j,\varphi\rangle_{\partial\Omega} = \int_\Omega a(u)\nabla u\cdot \nabla \varphi + b(x,u) \cdot\nabla\varphi +c(x,u,\nabla u)\varphi \d x
\end{align}
for all $\phi \in H^1(\Omega)$. 
With similar reasoning as in the parabolic case, we then obtain
\begin{theorem} \label{thm:3}
Let Assumption~\ref{ass:1} hold and $a_i$, $b_i$, $c_i$, $i=1,2$ satisfy Assumption~\ref{ass:2e}.
For any $g \in G$ let $u_i(g)$ denote 
weak solutions in the sense of Assumption~\ref{ass:3e} and let $j_i(g)$ be the corresponding Neumann fluxes. Then 
\begin{align}
a_1 \not\equiv a_2 \text{ on } (\underline g,\overline g) 
\implies j_1(g) \not\equiv j_2(g) \text{ on } \Gamma_M \text{ for some } g \in G.
\end{align}
If, in addition, the weak solution $u(g)$ is unique for any $g \in G$ and 
satisfies $\underline g \le u(g) \le \overline u(g)$, 
and if $b_1 \equiv b_2$ and $c_1 \equiv c_2$, then the reverse implication holds as well.
\end{theorem}
Let us note that the corresponding problem with $b\equiv0$ and $c(x,u,\nabla u)=c(x)u$ has already been treated in \cite{EPS2014a}, where knowledge of $a$ allowed us to determine coefficient $c\geq 0$ in a second step; see also \cite{Isakov06} for further results in this direction and the comments in Section~\ref{sec:discussion}.

\subsection{Extensions} \label{sec:extensions}

Before we close this section, let us briefly discuss some possible extensions concerning our assumptions on the parameters and the domain. 

\subsubsection{Regularity of the solution, the coefficients, and the boundary data}

A brief inspection of our estimates shows that the assumptions on the coefficients can be relaxed considerably, e.g., only some integrability or growth conditions for the coefficient governing the lower order terms are required. Also the regularity requirement on the Dirichlet data can be relaxed  and estimates for the solution in $L^p$ spaces may be used. We leave the details of such generalizations to the reader. 

\subsubsection{Assumptions on the geometry}
The usual localization argument allows us to deal also with the case that $\Gamma_M$ is not flat but given as the graph of a $C^1$ function. For sufficiently small $\eps_0$, we may assume that $\Gamma_M$ is almost flat and $n(x)=n(\bar x) + o(1)$, which is enough to prove the required estimates with slight modification of the proofs given above. We again leave the details to the reader.

\subsubsection{Boundary conditions}
In our proofs we only require local control of the Dirichlet boundary values on $\Gamma_M$.
Other types of boundary conditions, e.g., of Neumann or Robin type or even nonlinear conditions, can therefore be prescribed on the inaccessible part $\partial\Omega \setminus \Gamma_M$.
In particular, we do not need knowledge of any boundary data on $\partial\Omega \setminus \Gamma_M$. Hence, the data really required for our uniqueness results consists of 
$$
\{(g_{\mid\Gamma_M\times(0,T)},j(g)_{\mid\Gamma_M\times(0,T)})\}_{g \in G},
$$ 
which is sometimes referred to as local Dirichlet-to-Neumann map \cite{KenigSalo2014}.
A quick inspection of our proofs reveals that the additional terms coming from the boundary conditions on the inaccessible part can again be treated as lower order terms and therefore do not influence the validity of our results.

\section{Applications}\label{sec:examples}

The uniqueness results of the previous sections have many practically relevant applications. 
For illustration, we now discuss discuss in some detail problems arising in bioheat transfer, in chemotaxis or urban crime modeling, and in semiconductor device simulation. 
Particular emphasis will be put on the verification of our assumptions 
in the context of these applications.

\subsection{Nonlinear heat transfer}
Heat transfer in biological tissue, for instance in the liver, can be modeled by the following quasilinear heat equation \cite{Rojczyk2015}
\begin{align}\label{eq:heat}
 \partial_t u - \div ( a(u)\nabla u) + c(u_b-u) = 0\text{ on } \Omega \times (0,T).
\end{align}
Here $u$ describes the unknown temperature of the tissue while the blood temperature $u_b$ and the coefficient $c$ are assumed to be known. The heat conduction coefficient $a=a(u)$ is related to the material properties and is, in general, unknown. To close the system, we supplement the problem with the following initial and boundary values
\begin{align}
 u(x,0) &= u_0(x) \text{ in } \Omega,\\
 u(x,t) &= g(x,t) \text{ on } \partial\Omega\times (0,T). \label{eq:heat_bc}
\end{align}
We assume that the Dirichlet data $g \in G$ satisfy the conditions of Assumption~\ref{ass:3} 
with $\underline g = 0$ and $\overline g = u_b$.
As a consequence of maximal regularity results \cite[Theorem 1.1]{Amann2005},
one can show global existence of a unique weak solution $u \in L^2(0,T;H^1(\Omega)) \cap C^0([0,T];L^2(\Omega))$ for any initial value $u_0 \in L^2(\Omega)$, 
provided that $a$ and $u_0$ satisfy the conditions of Assumption~\ref{ass:2}. 
In addition, the uniform bounds
\begin{align*}
 \|u\|_{L^2(0,T;H^1(\Omega))} \le C \left( \|u_0\|_{L^2(\Omega)} + \|g\|_{L^2(0,T;H^{1/2}(\partial\Omega))} \right).
\end{align*}
hold and $0 \le u \le u_b$ on $\Omega\times (0,T)$ if also $0\leq u_0\leq u_b$. 
The last assertion follows from maximum principles which can be applied once solvability is known, since then the coefficient $a(u)$ can be treated as a space-dependent coefficient and the equation becomes linear.

These considerations show that Assumptions~\ref{ass:3} and \ref{ass:4} hold.
By application of Theorem~\ref{thm:2}, we thus obtain the following result.
\begin{corollary} 
Let $a_i:\RR\to\RR$, $i\in\{1,2\}$, be two functions satisfying the above conditions, 
and let $u_i(0)=u_0\in L^2(\Omega)$ be given initial data.
Then 
\begin{align*}
a_1 \equiv a_2 \text{ on } (0,u_b) \quad \Longleftrightarrow \quad j_1(g)\equiv j_2(g) \text{ for all } g \in G.
\end{align*}
\end{corollary}
In principle, it is again sufficient to know $j_i(g)$ for all $g \in G$ on a part $\Gamma_M$ of the boundary; see Section~\ref{sec:extensions}. 


\subsection{Coupled nonlinear drift-diffusion}
Coupled systems of  non-linear drift-diffusion equations appear in many applications, e.g., in chemotaxis \cite{Perthame2007} or in modeling and prediction of urban crime \cite{Short2008}. 
Here we consider a parabolic elliptic system of the form
\begin{align}\label{eq:chemo1non}
 \partial_t  u -  \div (a( u)\nabla  u + b(u)\nabla V)& = 0\text{ in } \Omega\times(0,T),\\\label{eq:chemo2}
  -\Delta V + V  &= h(u)\text{ in } \Omega\times(0,T).
\end{align}
The function $a(u)$ depends on the system at hand and is usually unknown.
The same is in principle true also for the coefficients $b$ and $h$, which may however be at least partially determined, provided that $a$ is known and that the measurements are sufficiently rich; see \cite{EPS2015a} and Section~\ref{sec:discussion}.
To complete the description of the problem, we further require the boundary conditions
\begin{align}
   u &= g \text{ on } \partial\Omega\times(0,T),\\
   \partial_n V & =  0\text{ on } \partial\Omega\times(0,T),
\end{align}
and we assume knowledge of the initial state for the first variable 
\begin{align}
  u(x,0) &=  u_0(x) \text{ in } \Omega.
\end{align}
Let us now turn to the verification of our assumptions. 
It is well known that for $b(u)=u$, the above system may exhibit blow-up in finite time.
This can be prevented by requiring
\begin{align}\label{eq:ASSCHEMO1}
b(0) = b(1) = 0,
\end{align}
and $u_0 \in C^2(\overline{\Omega})$ with $0\le  u_0  \le 1$ and $g\in G \cap C^\infty(\overline\Omega \times [0,T])$, where we set $\underline g=0$ and $\bar g=1$.
Furthermore, we assume $h \in W^{1,\infty}(\RR)$.
 Using similar arguments as in \cite[Theorem 3.1]{EPS2015a},  Assumptions~\ref{ass:1} and \ref{ass:2} ensure the global existence of a unique solution $(u,V)$ with $u \in L^p(0,T;H^{1}(\Omega))\cap H^{1}(0,T;L^2(\Omega))$ and $V\in L^\infty(0,T;W^{2,p}(\Omega))$. Moreover, the following a priori bound holds
  \begin{align*}
     \|u\|_{L^2(0,T;H^{1}(\Omega))}  \leq C\big( \|u_0\|_{H^{1}(\Omega)} + \|g\|_{H^1(\Omega_T)}\big). 
  \end{align*}
In addition we can apply \cite[Lemma 3.2]{EPS2015a} to ensure that 
\begin{align*}
 0 \le u(x,t) \le 1 \text{ a.e. on } \Omega\times (0,T),
\end{align*}
and thus Assumption~\ref{ass:4} is also satisfied.
%
We can now apply Theorem~\ref{thm:1} by considering only equation \eqref{eq:chemo1non} 
and treating $V$ as an unknown term of lower order. This leads to the following uniqueness result. 
\begin{corollary} 
Let $a_i,\, b_i:\RR\to\RR$, $i\in\{1,2\}$, be functions satisfying Assumption~\ref{ass:2}, let $h_i\in W^{1,\infty}(\RR)$, $i\in\{1,2\}$, with norm bounded by $C_A$ as in Assumption~\ref{ass:2},
and let $ u_{0,i} \in C^2(\overline{\Omega})$ with $0\leq u_{0,i}\leq 1$ be given. 
Then 
\begin{align*}
j_1(g) \equiv j_2(g) \text{ for all } g \in G \implies a_1 \equiv a_2 \text{ on } (0,1).
\end{align*}
If $b_i=b$, $h_i=h$ and $u_{i,0}=u_0$ are known, then the reverse implications holds true as well. 
\end{corollary}
Similar as in Theorem~\ref{thm:2}, the reverse direction follows directly from the uniqueness of the solution to the parabolic elliptic system for given parameters and initial values.

\subsection{Nonlinear Poisson-Nernst-Planck systems}
Our arguments can also be applied to more complicated systems of partial differential equations. 
To illustrate this, let us consider the Poisson-Nernst-Planck system which arises in models for semiconductors \cite{Markowich1990} and also describes ionic fluxes through biological and synthetic channels \cite{Eisenberg2007,Levitt1991}. 
Here we consider a system with non-linear diffusion coefficients which is relevant for high concentration densities \cite{Juengel1994}. The model equations then read
\begin{align}\label{eq:PNP1}
 \partial_t u &= \div ( a(u) \nabla u + u\nabla V) \text{ in }\Omega\times (0,T)\\ \label{eq:PNP2}
 \partial_t \tilde u &= \div ( \tilde a(\tilde u) \nabla \tilde u - \tilde u\nabla V) \text{ in }  \Omega\times (0,T)\\ \label{eq:PNP3}
 -\Delta V &= u-\tilde u + \xi \text{ in }  \Omega\times (0,T),
\end{align}
where $u$, $\tilde u$ are concentrations of positive and negative charges, $V$ is the electric potential, and $\xi$ is a given charge distribution. 
The system is complemented with two initial conditions
\begin{align}
 u(x,0) = u_0(x) \text{ and } \tilde u(x,0) = \tilde u_0(x),
\end{align}
and we prescribe 
\begin{align}
u = g,\quad \tilde u = \tilde g,\text{ on } \Gamma_D\times (0,T), \label{eq:PNPbc2}
\end{align}
on $\Gamma_D \subset \partial \Omega$ which models open ends of the domain connected to respective charge reservoirs. We choose $g,\tilde g \in G$ satisfying $ 0=\underline g \le g,\tilde g \le \overline g$ for appropriate $\overline g>0,$ depending on $T$ and $\|\xi\|_{L^\infty(\Omega)}$. 
Homogeneous Neumann boundary conditions for $u$ and $\tilde u$ are used on $\Gamma_N=\partial\Omega \setminus \Gamma_D$ and appropriate mixed boundary conditions are prescribed for the potential $V$. 

In order to ensure well-posedness of the Poisson-Nernst-Planck system we require additional conditions to hold, namely \cite[(H1)--(H5)]{Juengel1994}; let us refer to \cite[Sec 1.7]{Juengel1994} for a detailed discussion under what conditions on the geometry these assumptions are satisfied.
Assuming \cite[(H1)--(H5)]{Juengel1994}, existence of a unique weak solutions $(u,\tilde u,V) \in [L^2(0,T;H^1(\Omega)\cap L^\infty(\Omega))]^2 \times L^\infty(0,T;W^{2,p}(\Omega))$ has been derived in \cite[Theorem 2.3, 2.4]{Juengel1994}.
In addition, one can obtain uniform bounds for the solution depending only on the bounds for the problem data.  

With similar reasoning as in the previous example, we can consider $\nabla V$ as an unknown lower order term.
In order to separate the influence of $a$ and $\tilde a$ in the flux $j(g_u,g_w)$, we choose either $g_u$ or $g_w$ constant which implies that one of the corresponding primitive functions $A$ or $\tilde A$ vanishes.
\begin{corollary} 
Let the assumptions of \cite[Theorem 2.3, 2.4]{Juengel1994} on the domain, the coefficients,
and the initial data be valid. Furthermore, let $u_i(g,\tilde g)$, $\tilde u_i(g,\tilde g)$, $i=1,2$ denote the weak solutions corresponding to parameters $a_i,\tilde a_i$, and let $j_i(g,\tilde g)$ denote the corresponding fluxes. Then 
\begin{align*}
j_1(g,\tilde g) \equiv j_2(g,\tilde g)  \text{ on } \Gamma_D \times (0,T) \text{ for all } g,\,\tilde g \in G
&\implies a_1 \equiv a_2\text{ and } \tilde a_1 \equiv \tilde a_2  \text{ on } (\underline g,\overline g).
\end{align*}
\end{corollary}
It should be clear now that with similar reasoning one can determine diffusion coefficients in rather general systems of parabolic elliptic type.

\section{Discussion} \label{sec:discussion}

In this paper, we considered the identification of diffusion coefficient functions $a(t,u)$ and $a(u)$ in parabolic and elliptic partial differential equations. Using rather general arguments based on singular functions, we were able to establish uniqueness from local observation of the Dirichlet-to-Neumann map even in the presence of unknown lower order terms. This allows to apply our results to rather general problems, which was demonstrated in examples. 

Knowledge of the leading order coefficient can eventually be used to obtain uniqueness also for the coefficients in lower order terms in a second step. 
For illustration of the main idea, let us consider an elliptic problem of the form
\begin{align*} 
-\div (a(u) \nabla u) + c(x,u,\nabla u) &= 0 \qquad \text{in } \Omega,\\
u &= g \qquad \text{on } \partial\Omega.
\end{align*}
The knowledge of the diffusion coefficient $a(u)$ allows to introduce a new variable $w=A(u)$, 
with $A(u)=\int_0^u a(g) dg$, and to transform the problem equivalently into 
\begin{align*}
-\Delta w + \tilde c(x,w,\nabla w) &= 0 \qquad \text{in } \Omega, \\
w &= \tilde g \qquad \text{on } \partial\Omega,
\end{align*}
with $\tilde g = A(g)$ and $\tilde c(x,w,\nabla w) = c(x,A^{-1}(w),(1/a(A^{-1}(w))\nabla w)$. 
The transformation $w=A(u)$ is a diffeomorphism and therefore also the Dirichlet-to-Neumann map 
transforms equivalently. Some special problems of this form have already been treated successfully: The identification of $c(x,w)$ has been addressed successfully by Isakov and Sylvester~\cite{Isakov1994} and uniqueness for $c(x,\nabla w)$ has been established by Sun~\cite{Sun04}. The uniqueness of a coefficient $c(w,\nabla w)$ is shown in \cite{Isakov01}.
The results of this paper may therefore be valuable as one basic ingredient for the proof of uniqueness of several parameters in rather general parabolic and elliptic problems. 

\section*{Acknowledgements}
HE acknowledges support by DFG via Grant IRTG~1529, GSC~233, TRR~154, and Eg-331/1-1. 
The work of JFP was supported by the DFG via Grant Pi-1073/1-2. The authors would like to thank Michael Winkler (Paderborn) and Ansgar J\"ungel (Vienna) for useful hints to literature.

\bibliographystyle{abbrv}
\bibliography{bib}

\end{document}